\documentclass[12pt,twoside]{amsart}
\usepackage{amssymb,amscd,amsxtra,calc,mathrsfs}
\usepackage{amsmath}
\usepackage{xypic}

\title[K-polystability of log del Pezzo pairs]{On K-polystability for log del Pezzo pairs 
of Maeda type}
\author{Kento Fujita} 
\date{\today}
\address{Department of Mathematics, Graduate School of Science, Osaka University, 
Toyonaka, Osaka 560-0043, Japan}
\email{fujita@math.sci.osaka-u.ac.jp}

\newcommand{\pr}{\mathbb{P}}

\newcommand{\Z}{\mathbb{Z}}
\newcommand{\Q}{\mathbb{Q}}
\newcommand{\R}{\mathbb{R}}
\newcommand{\C}{\mathbb{C}}
\newcommand{\F}{\mathbb{F}}

\newcommand{\A}{\mathbb{A}}
\newcommand{\G}{\mathbb{G}}

\newcommand{\Supp}{\operatorname{Supp}}

\newcommand{\Pic}{\operatorname{Pic}}

\newcommand{\Aut}{\operatorname{Aut}}
\newcommand{\Proj}{\operatorname{Proj}}

\newcommand{\ord}{\operatorname{ord}}

\newcommand{\vol}{\operatorname{vol}}

\newcommand{\PGL}{\operatorname{PGL}}

\newcommand{\diag}{\operatorname{diag}}

\newcommand{\coeff}{\operatorname{coeff}}
\newcommand{\MOD}{\operatorname{mod}}

\newcommand{\sC}{\mathcal{C}}

\newcommand{\sO}{\mathcal{O}}

\newcommand{\sX}{\mathcal{X}}

\newcommand{\sF}{\mathcal{F}}

\newtheorem{thm}{Theorem}[section]
\newtheorem{lemma}[thm]{Lemma}
\newtheorem{proposition}[thm]{Proposition}
\newtheorem{corollary}[thm]{Corollary}

\theoremstyle{definition}
\newtheorem{definition}[thm]{Definition}
\newtheorem{remark}[thm]{Remark}

\newtheorem{example}[thm]{Example}
\newtheorem*{ack}{Acknowledgments}

\begin{document}

\maketitle 

\begin{abstract}
We give an algebraic proof for which log del Pezzo pairs of Maeda type are 
K-polystable or not. 
If the base field is the complex number field, then the result is already 
known by Li and Sun. 
\end{abstract}

\setcounter{tocdepth}{1}
\tableofcontents

\section{Introduction}\label{intro_section}

We work over an arbitrary algebraically closed field $\Bbbk$ with the characteristic zero. 
Let $X$ be a \emph{Fano manifold}, that is, $X$ is a smooth projective variety 
over $\Bbbk$ such that the anti-canonical divisor $-K_X$ is ample. We are interested in 
the problem whether $X$ is \emph{K-polystable} or not. In fact, if $\Bbbk$ is equal to the 
complex number field $\C$, then 
K-polystability of $X$ is known to be equivalent to the existence of K\"ahler-Einstein 
metrics on $X$ thanks to the works \cite{tian, don, stoppa, B, CDS1, CDS2, CDS3, tian2} 
and references therein. It is natural to consider K-polystability for not only 
Fano manifolds but also \emph{log Fano pairs} $(X, \Delta)$ 
(see Definition \ref{logfano_dfn} \eqref{logfano_dfn4}). 
However, in general, it is difficult to test K-polystability 
purely algebraically. Recently, Li, Wang and Xu in \cite[Theorem 1.4]{LWX} 
gave a purely algebraic proof for which toric log Fano pairs are K-polystable or not. 
However, when a log Fano pair is not a toric pair, it is difficult to test K-polystability. 
See also Remark \ref{PPKss_rmk}. 

In this article, we mainly consider K-polystability of \emph{log del Pezzo pairs}, 
that is, log Fano pairs of dimension two. 
The purpose of this article is to give an algebraic proof for K-polystability of 
the log del Pezzo pair $(\pr^2, \delta C)$, where 
$\delta$ is a non-negative rational number with $\delta<3/4$ and $C\subset\pr^2$ 
is a smooth conic, and the log del Pezzo pair $(\pr^1\times\pr^1, \delta C)$, where 
$\delta$ is a non-negative rational number with $\delta<1/2$ and 
$C\subset\pr^1\times\pr^1$ is the diagonal. 

\begin{thm}[{cf.\ \cite[Example 3.12]{LS}}]\label{mainthm}
\begin{enumerate}
\renewcommand{\theenumi}{\arabic{enumi}}
\renewcommand{\labelenumi}{(\theenumi)}
\item\label{mainthm1} 
Assume that $C\subset\pr^2$ is a smooth conic and let $\delta\in[0,1)\cap\Q$. 
Then the log del Pezzo pair $(\pr^2, \delta C)$ is K-polystable $($resp., K-semistable$)$ 
if and only if $\delta<3/4$ $($resp., $\delta\leq 3/4$$)$. 
\item\label{mainthm2} 
Assume that $C\subset\pr^1\times\pr^1$ is the diagonal and let $\delta\in(0,1)\cap\Q$. 
Then the log del Pezzo pair $(\pr^1\times\pr^1, \delta C)$ is 
K-polystable $($resp., K-semistable$)$ 
if and only if $\delta<1/2$ $($resp., $\delta\leq 1/2$$)$. 
\end{enumerate}
\end{thm}

If $\Bbbk=\C$, then the above result is known by \cite[Example 3.12]{LS} and 
\cite[Theorem 4.8]{B}. 
We emphasize that, our proof is based on the work \cite{pltK}, 
purely algebraic, direct and easy. Moreover, in Theorem \ref{mainthm} \eqref{mainthm1}, 
we give a very easy and purely algebraic proof for K-polystability of $\pr^2$. 
For the proof of Theorem \ref{mainthm} \eqref{mainthm2}, we use the fact 
$\pr^1\times\pr^1$ is K-semistable. We can prove this fact purely algebraically 
(see \cite{kempf, li, blum, BJ}). 

As an immediate consequence of Theorem \ref{mainthm}, we get an algebraic proof 
for the classification of K-polystable log del Pezzo pairs of Maeda type. 
A pair $(X, \Delta)$ is said to be a \emph{log del Pezzo pair of Maeda type} if 
$X$ is a smooth projective surface and $\Delta$ is a nonzero effective $\Q$-divisor 
on $X$ such that $D:=\Supp\Delta$ is simple normal crossing and both 
$-(K_X+\Delta)$ and $-(K_X+D)$ are ample.

\begin{corollary}\label{maincor}
Let $(X, \Delta)$ be a log del Pezzo pair of Maeda type. 
Then $(X, \Delta)$ is K-polystable $($resp., K-semistable$)$ if and only if 
\begin{itemize}
\item
$(X, \Delta)$ is isomorphic to $(\pr^2, \delta C)$ with $C$ 
a smooth conic and $\delta<3/4$
$($resp., $\delta\leq 3/4$$)$, or
\item
$(X, \Delta)$ is isomorphic to $(\pr^1\times\pr^1, \delta C)$ with $C$ the diagonal 
and $\delta<1/2$ $($resp., $\delta\leq 1/2$$)$.
\end{itemize}
\end{corollary}

For the proof, we use Maeda's classification result \cite{maeda}. 
In general, Cheltsov and Rubinstein gave a question in \cite{CR} that which 
asymptotically log del Pezzo pairs (see \cite[Definition 1.1]{CR}) are K-polystable or not. 
In order to consider the question, it is important to establish techniques 
to test K-polystability of log del Pezzo pairs. The above log del Pezzo pairs 
$(X, \Delta)$ in Theorem \ref{mainthm} are no longer uniformly K-stable 
(see Theorems \ref{F_thm} and \ref{FF_thm}). 
Hence we cannot apply the techniques to evaluate the \emph{delta invariants} 
introduced in \cite{FO, BJ} in order to show K-polystability. 
The proof of Theorem \ref{mainthm} will be important 
to answer the question of Cheltsov and Rubinstein. 

This article is organized as follows. In \S \ref{K_section}, we give the definitions for 
K-polystability and K-semistability of log Fano pairs. The definitions are not of original 
form in \cite{tian, don}. Moreover, we see several numerical properties of the 
invariants $\hat{\beta}_{(X, \Delta)}(F)$ for log del Pezzo pairs $(X, \Delta)$ and 
of dreamy prime 
divisors $F$ over $(X, \Delta)$. In \S \ref{surface_section}, we see basic properties 
for exceptional prime divisors over smooth surfaces. Moreover, we see that there exists 
a non-dreamy prime divisor over $\pr^2$. In \S \ref{product_section}, we discuss 
product-type prime divisors over $(\pr^2, \delta C)$ and $(\pr^1\times\pr^1, \delta C)$. 
In \S \ref{P_section}, we prove Theorem \ref{mainthm} \eqref{mainthm1}; 
in \S \ref{PP_section}, we prove Theorem \ref{mainthm} \eqref{mainthm2}; 
in \S \ref{maeda_section}, we prove Corollary \ref{maincor}.

\begin{ack}
This work was supported by JSPS KAKENHI Grant Number 18K13388.
\end{ack}

For the minimal model program, we refer the readers to \cite{KoMo}. 
For a birational map $X\dashrightarrow X'$ between normal projective varieties and 
for a $\Q$-divisor $\Delta$ on $X$, the strict 
transform of $\Delta$ on $X'$ is denoted by $\Delta^{X'}$. 
Moreover, for a prime divisor $E$ on $X$, the coefficient of $\Delta$ at $E$ is 
denoted by $\coeff_E\Delta$. 

For the toric geometry, we refer the readers to \cite{CLS}. 
In this article, we consider only $2$-dimensional toric varieties. We always fix the lattice 
$N:=\Z^{\oplus 2}$ of rank $2$ and set $N_\R:=N\otimes_\Z\R$.

\section{K-stability of log Fano pairs}\label{K_section}

We recall K-polystability and K-semistability of log Fano pairs in \cite{pltK, ha}. 
The definition is equivalent to the original one \cite{tian, don} by the 
works \cite{li, vst, pltK, ha}. 

\begin{definition}\label{logfano_dfn}
Let $(X, \Delta)$ be a \emph{log pair}, that is, $X$ is a normal variety and $\Delta$ is 
an effective $\Q$-divisor on $X$ such that $K_X+\Delta$ is $\Q$-Cartier. Let $F$ be a 
prime divisor \emph{over} $X$, that is, there exists a resolution $\pi\colon\tilde{X}\to X$ 
such that $F$ is a prime divisor \emph{on} $\tilde{X}$. 
\begin{enumerate}
\renewcommand{\theenumi}{\arabic{enumi}}
\renewcommand{\labelenumi}{(\theenumi)}
\item\label{logfano_dfn1}
We set 
\[
A_{(X, \Delta)}(F):=1+\coeff_F\left(K_{\tilde{X}}-\pi^*(K_X+\Delta)\right). 
\]
The center (i.e., the image) of $F$ on $X$ is denoted by $c_X(F)$. 
We recall that the pair $(X, \Delta)$ is said to be \emph{klt} if $A_{(X, \Delta)}(F)>0$ 
for any prime divisor $F$ over $X$. 
\item\label{logfano_dfn2}
(\cite{ishii})
The $F$ is said to be \emph{primitive} over $X$ if there exists a projective birational 
morphism $\sigma\colon Y\to X$ (called the \emph{extraction of $F$}) with $Y$ normal 
such that $-F$ is a $\sigma$-ample $\Q$-Cartier divisor on $Y$. 
\item\label{logfano_dfn3}
(\cite{shokurov, prokhorov})
The $F$ is said to be \emph{plt-type over $(X, \Delta)$} if $F$ is primitive over $X$ and 
$(Y, \Delta_Y+F)$ is plt, where $\sigma\colon Y\to X$ is the extraction of $F$ and 
$\Delta_Y$ is the $\Q$-divisor on $Y$ given by the equation 
\[
K_Y+\Delta_Y+\left(1-A_{(X, \Delta)}(F)\right)F=\sigma^*(K_X+\Delta).
\]
\item\label{logfano_dfn4}
The pair $(X, \Delta)$ is said to be a \emph{log Fano pair} if $(X, \Delta)$ is a projective 
klt pair such that $-(K_X+\Delta)$ is an ample $\Q$-divisor on $X$. If moreover the 
dimension of $X$ is equal to $2$, then we call it a \emph{log del Pezzo pair}. 
\end{enumerate}
\end{definition}

\begin{definition}[{see \cite{li, vst, pltK, ha}}]\label{beta_dfn}
Let $(X, \Delta)$ be an $n$-dimensional log Fano pair and set $L:=-(K_X+\Delta)$. 
Take any prime divisor $F$ over $X$ and let us fix a resolution $\pi\colon\tilde{X}\to X$ 
such that $F$ is a prime divisor on $\tilde{X}$. 
\begin{enumerate}
\renewcommand{\theenumi}{\arabic{enumi}}
\renewcommand{\labelenumi}{(\theenumi)}
\item\label{beta_dfn1}
For any $x\in\R_{\geq 0}$ and for any $r\in\Z_{\geq 0}$ with $rL$ Cartier, let 
$H^0(X, rL-xF)$ be the subspace of $H^0(X, rL)$ given by 
\[
H^0(X, rL-xF):=H^0\left(\tilde{X}, \pi^*(rL)\left(\lfloor-xF\rfloor\right)\right)
\subset H^0\left(\tilde{X}, \pi^*(rL)\right)
\]
under the natural identity $H^0(X, rL)=H^0\left(\tilde{X}, \pi^*(rL)\right)$.
\item\label{beta_dfn2}
For any $x\in\R_{\geq 0}$, we set 
\[
\vol(L-xF):=\limsup_{r\to\infty}\frac{\dim_\Bbbk H^0(X, rL-rxF)}{r^n/n!}.
\]
We set 
\[
\tau(F):=\sup\{x\in\R_{\geq 0}\,\,|\,\,\vol(L-xF)>0\}.
\]
Moreover, if $F$ is primitive over $X$, then we set 
\[
\varepsilon(F):=\max\{x\in\R_{\geq 0}\,\,|\,\,\sigma^*L-xF\text{ is nef on }Y\}, 
\]
where $\sigma\colon Y\to X$ is the extraction of $F$. 
Obviously, we have $\varepsilon(F)\leq \tau(F)$. 
\item\label{beta_dfn3}
We set 
\[
\hat{\beta}_{(X, \Delta)}(F):=
1-\frac{\int_0^{\infty}\vol(L-xF)dx}{A_{(X, \Delta)}(F)\cdot(L^{\cdot n})}.
\]
\item\label{beta_dfn4}
The $F$ is said to be \emph{dreamy over $(X, \Delta)$} if the graded $\Bbbk$-algebra
\[
\bigoplus_{k,\, j\in\Z_{\geq 0}}H^0(X, krL-jF)
\]
is finitely generated over $\Bbbk$ for some $r\in\Z_{>0}$ with $rL$ Cartier. 
\item\label{beta_dfn5}
The $F$ is said to be \emph{product-type over $(X, \Delta)$} if there exists a 
$1$-parameter subgroup $\rho\colon \G_m\to \Aut(X, \Delta)$ of $\Aut(X, \Delta)$ 
such that the divisorial valuation $\ord_F\colon \Bbbk(X)^*\to\Z$ is equal to 
the composition 
\[
\Bbbk(X)^*\xrightarrow{\rho^*} \Bbbk(X)(t)^*\xrightarrow{\ord_{(t^{-1})}} \Z,
\]
where 
\[
\Aut(X, \Delta):=\{\theta\in\Aut(X)\,\,|\,\,\theta_*\Delta=\Delta\}\subset\Aut(X),
\]
and $\rho^*\colon \Bbbk(X)\to\Bbbk(X)(t)$ is given by the natural morphism 
$\rho\colon \G_m\times X\to X$. 
\end{enumerate}
\end{definition}

\begin{remark}\label{beta_rmk}
\begin{enumerate}
\renewcommand{\theenumi}{\arabic{enumi}}
\renewcommand{\labelenumi}{(\theenumi)}
\item\label{beta_rmk1}
The above definitions are not depend on the choice of the morphism 
$\pi\colon\tilde{X}\to X$. 
\item\label{beta_rmk2}
The function $\vol(L-xF)$ is continuous and non-increasing over $x\in[0, \infty)$ 
by \cite{L1, L2}. Moreover, by \cite[Theorem A]{BFJ}, $\vol(L-xF)$ is $\sC^1$ 
over $x\in[0, \tau(F))$.
\item\label{beta_rmk3}
By \cite[Proposition 2.4]{ishii}, the extraction of $F$ is unique if exists. 
\item\label{beta_rmk4}
If $F$ is product-type over $(X, \Delta)$, then $F$ is dreamy over $(X, \Delta)$ by 
\cite[Proposition 3.10]{ha}. 
If $F$ is dreamy over $(X, \Delta)$, then $F$ is primitive over $X$ 
by \cite[Remark 1.3 (1)]{pltK}. 
\end{enumerate}
\end{remark}

\begin{definition}\label{K_dfn}
Let $(X, \Delta)$ be a log Fano pair. 
\begin{enumerate}
\renewcommand{\theenumi}{\arabic{enumi}}
\renewcommand{\labelenumi}{(\theenumi)}
\item\label{K_dfn1}
The pair $(X, \Delta)$ is said to be \emph{K-semistable} (resp., \emph{K-stable}) 
if $\hat{\beta}_{(X, \Delta)}(F)\geq 0$ (resp., $>0$) 
for any dreamy prime divisor $F$ over $(X, \Delta)$. 
\item\label{K_dfn2}
The pair $(X, \Delta)$ is said to be \emph{K-polystable} if K-semistable, 
and a dreamy prime divisor $F$ over $(X, \Delta)$ satisfies that 
$\hat{\beta}_{(X, \Delta)}(F)=0$ 
only if $F$ is a product-type over $(X, \Delta)$. 
\item\label{K_dfn3}
The pair $(X, \Delta)$ is said to be \emph{uniformly K-stable} 
if there exists $\varepsilon>0$ such that $\hat{\beta}_{(X, \Delta)}(F)\geq \varepsilon$ 
for any dreamy prime divisor $F$ over $(X, \Delta)$. 
\end{enumerate}
\end{definition}

\begin{remark}\label{K_rmk}
\begin{enumerate}
\renewcommand{\theenumi}{\arabic{enumi}}
\renewcommand{\labelenumi}{(\theenumi)}
\item\label{K_rmk1}
By the works \cite{li, vst, pltK, ha}, the notions of K-semistability, K-polystability, 
K-stability and uniform K-stability are equivalent to the original one in \cite{tian, don, LX}. 
\item\label{K_rmk2}
It is known that K-semistability (resp., uniform K-stability) 
of $(X, \Delta)$ is equivalent to the 
condition $\hat{\beta}_{(X, \Delta)}(F)\geq 0$ (resp., 
$\hat{\beta}_{(X, \Delta)}(F)\geq \varepsilon$) 
for any prime divisor $F$ over $X$. See \cite{li, vst, pltK}. 
\end{enumerate}
\end{remark}

We recall the following: 

\begin{proposition}\label{tau_prop}
Let $(X, \Delta)$ be an $n$-dimensional log Fano pair and let $F$ be a prime divisor 
over $X$. If $\tau(F)\leq A_{(X, \Delta)}(F)$, then we have the inequality 
\[
\hat{\beta}_{(X, \Delta)}(F)\geq \frac{1}{n+1}.
\]
\end{proposition}

\begin{proof}
Follows immediately from \cite[Proposition 2.1]{pltK}. 
\end{proof}

The following proposition is essential in \S \ref{PP_section}. 

\begin{proposition}\label{convex_prop}
Let $(X, \Delta)$ be an $n$-dimensional log Fano pair, let $L:=-(K_X+\Delta)$, and 
let $F$ be a prime divisor over $X$. Set $f(x):=\vol(L-xF)$ for $x\in \R_{\geq 0}$. 
Then, for any $0\leq x<y\leq \tau(F)$, we have the following inequality 
\[
f(y)\leq f(x)\left(\frac{y-x}{n}\frac{f'(x)}{f(x)}+1\right)^n.
\]
In particular, if $x>0$, then we have 
\[
\tau(F)\leq x+\frac{n f(x)}{-f'(x)}.
\]
\end{proposition}

\begin{proof}
We may assume that $x>0$. For any $c\in[0$, $x)$, we have 
\[
f(x)^{1/n}\geq \frac{y-x}{y-c}f(c)^{1/n}+\frac{x-c}{y-c}f(y)^{1/n}
\]
by the log-concavity of the volume functions (see, e.g., \cite{LM}). Thus we have 
\begin{eqnarray*}
\frac{1}{y-x}f(y)^{1/n}&\leq & \lim_{c\nearrow x}\frac{\frac{y-x}{y-c}f(c)^{1/n}
-f(x)^{1/n}}{c-x}=\frac{\partial}{\partial c}\bigg|_{c=x}\left(\frac{y-x}{y-c}f(c)^{1/n}\right)\\
&=&\frac{f(x)^{1/n}}{y-x}\left(\frac{f'(x)}{n f(x)}(y-x)+1\right).
\end{eqnarray*}
When $0<x< \tau(F)$, we know that $f'(x)<0$ (see \cite[Corollary 4.27]{LM} for 
example). 
Hence we get the assertion. 
\end{proof}

From now on, let us assume that $(X, \Delta)$ is a log del Pezzo pair with 
$\rho(X)=1$, where $\rho(X)$ is the Picard number of $X$. 
By \cite[Proposition 4.11]{KoMo}, $X$ is $\Q$-factorial. Take any dreamy exceptional 
prime divisor $F$ over $(X, \Delta)$. By Remark \ref{beta_rmk} \eqref{beta_rmk4}, $F$ is 
primitive over $X$. Let $\sigma\colon Y\to X$ be the extraction of $F$. Then 
$Y$ is $\Q$-factorial by \cite[Remark 2.2 (i)]{prokhorov}. Moreover, by 
\cite[Theorem 4.2]{KKL}, we have $\varepsilon(F)$, $\tau(F)\in\Q_{>0}$ and 
$\sigma^*L-\varepsilon(F)F\,\,(\not\sim_\Q 0)$ induces a non-trivial morphism 
$\mu\colon Y\to Z$ with connected fibers and with $Z$ normal and $\rho(Z)=1$. 
If $\varepsilon(F)<\tau(F)$, then $\mu$ is birational; if $\varepsilon(F)=\tau(F)$, 
then $Z\simeq\pr^1$. 

\begin{definition}\label{diag_dfn}
The above diagram 
\[\xymatrix{
& Y  \ar[dl]_\sigma \ar[dr]^\mu & \\
X & & Z
}\]
is called \emph{the standard diagram with respects to $F$}. 
\end{definition}

We frequently use the following lemma: 

\begin{lemma}\label{hinpan_lem}
Let $G\subset Y$ be an irreducible curve. 
\begin{enumerate}
\renewcommand{\theenumi}{\arabic{enumi}}
\renewcommand{\labelenumi}{(\theenumi)}
\item\label{hinpan_lem1}
If $\mu_*G=0$, then we have 
$\left(\left(\sigma^*L-\varepsilon(F)F\right)\cdot G\right)=0$.
\item\label{hinpan_lem2}
If $\mu_*G\neq 0$, then we have 
$\left(\left(\sigma^*L-\tau(F)F\right)\cdot G\right)\geq 0$.
\end{enumerate}
\end{lemma}

\begin{proof}
\eqref{hinpan_lem1}
The assertion is obvious since 
the $\Q$-divisor $\sigma^*L-\varepsilon(F)F$ is the pullback of a $\Q$-divisor on $Z$. 

\eqref{hinpan_lem2}
If $\varepsilon(F)=\tau(F)$, then the assertion is trivial since $\sigma^*L-\varepsilon(F)F$ 
is nef. If $\varepsilon(F)<\tau(F)$, then $\sigma^*L-\tau(F)F$ is $\Q$-linearly 
equivalent to some positive multiple of the $\mu$-exceptional curve. Thus the assertion 
follows. 
\end{proof}

The following lemma is proved as in the case with the proof of \cite[Claim 4.3]{pltK}. 

\begin{lemma}\label{beta_lem}
Let us set $L:=-(K_X+\Delta)$ and 
\[
f_0:=-\frac{1}{(F^{\cdot 2})_Y}.
\]
Then we have 
$\varepsilon(F)\cdot\tau(F)=f_0\cdot(L^{\cdot 2})$ and 
\[
\hat{\beta}_{(X, \Delta)}(F)=1-\frac{\varepsilon(F)+\tau(F)}{3\cdot A_{(X, \Delta)}(F)}.
\]
\end{lemma}

\begin{proof}
We recall the proof of \cite[Claim 4.3]{pltK}. Note that 
\[
\vol(L-xF)=(L^{\cdot 2})-\frac{x^2}{f_0}
\]
for any $x\in[0, \varepsilon(F)]$. If $\varepsilon(F)=\tau(F)$, then the assertion is trivial 
since we know that $\vol(L-\tau(F)F)=0$. We may assume that $\varepsilon(F)<\tau(F)$. 
Then, since $\mu_*\sigma^*L$ and $\mu_*F$ are $\Q$-linearly proportional on $Z$, 
there exists $c_0\in\Q_{>0}$ such that, for any $x\in[\varepsilon(F), \tau(F)]$, we have 
$\vol(L-xF)=c_0(\tau(F)-x)^2$. When we apply Remark \ref{beta_rmk} \eqref{beta_rmk2} 
with $x=\varepsilon(F)$, we have 
\begin{eqnarray*}
(L^{\cdot 2})-\frac{\varepsilon(F)^2}{f_0}&=&c_0\left(\tau(F)-\varepsilon(F)\right)^2, \\
-\frac{2\varepsilon(F)}{f_0}&=&-2c_0\left(\tau(F)-\varepsilon(F)\right).
\end{eqnarray*}
Thus the assertion follows. 
\end{proof}

\section{Basic properties of surfaces}\label{surface_section}

In this section, we see basic properties for exceptional prime divisors on surfaces 
in order to prove Theorem \ref{mainthm}. 

\subsection{Sequences of monoidal transforms}

\begin{definition}\label{monoidal_dfn}
Let $X$ be a smooth surface and let $F$ be an exceptional prime divisor over $X$. 
We construct the sequence 
\[
\pi\colon\tilde{X}=X_m\to\cdots\to X_1\to X_0=X
\]
of monoidal transform (called \emph{the sequence of monoidal transforms with respects 
to $F$}) given by: 
\begin{enumerate}
\renewcommand{\theenumi}{\arabic{enumi}}
\renewcommand{\labelenumi}{(\theenumi)}
\item\label{monoidal_dfn1}
$X_0:=X$. 
\item\label{monoidal_dfn1}
If $F$ is a prime divisor on $X_i$, then we set $m:=i$, $\tilde{X}:=X_m$ and 
we stop the construction. 
\item\label{monoidal_dfn1}
If $F$ is exceptional over $X_i$, then we set $p_{i+1}:=c_{X_i}(F)$, let 
$\pi_{i+1}\colon X_{i+1}\to X_i$ 
be the blowup along $p_{i+1}$ and let $E_{i+1}\subset X_{i+1}$ 
be the $\pi_{i+1}$-exceptional curve. 
\end{enumerate}
For any $1\leq i\leq m$, let $\tilde{E}_i\subset\tilde{X}$ be the strict transform of 
$E_i$ on $\tilde{X}$. Obviously, we have $p_{i+1}\in E_i$ for any $1\leq i\leq m-1$, and we have 
$\tilde{E}_m=F$. 
\end{definition}

\begin{definition}\label{star_dfn}
Under the notation in Definition \ref{monoidal_dfn}, we define the following notions: 
\begin{enumerate}
\renewcommand{\theenumi}{\arabic{enumi}}
\renewcommand{\labelenumi}{(\theenumi)}
\item\label{star_dfn1}
For any $2\leq i\leq m$, let us define $q(i)\in[0, i-2]\cap\Z$ as follows: 
\begin{itemize}
\item
If $p_i\not\in E_j^{X_{i-1}}$ for any $1\leq j\leq i-2$, then we set $q(i):=0$. 
\item
If $p_i\in E_j^{X_{i-1}}$ for some $1\leq j\leq i-2$, then we set $q(i):=j$. 
\end{itemize}
Since $E_1^{X_{i-1}},\dots,E_{i-2}^{X_{i-1}}$, $E_{i-1}$ are simple normal crossing and 
$p_i\in E_{i-1}$, the definition makes sense. 
\item\label{star_dfn2}
For any $0\leq i\leq m$, let us define the effective $\Z$-divisor $E_i^*$ on $\tilde{X}$ 
as follows: 
\begin{itemize}
\item
We set $E_0^*:=0$ and $E_1^*:=(\pi_2\circ\cdots\circ\pi_m)^*E_1$.
\item
For any $2\leq i\leq m$, we set 
\[
E_i^*:=E_{q(i)}^*+E_{i-1}^*+(\pi_{i+1}\circ\cdots\circ\pi_m)^*E_i.
\]
\end{itemize}
\item\label{star_dfn3}
We set 
$f:=\coeff_F E_m^*\in\Z_{>0}$.
\end{enumerate}
\end{definition}

\begin{lemma}\label{star_lem}
\begin{enumerate}
\renewcommand{\theenumi}{\arabic{enumi}}
\renewcommand{\labelenumi}{(\theenumi)}
\item\label{star_lem1}
For any $1\leq i$, $j\leq m$, we have $\left(E_i^*\cdot \tilde{E}_j\right)=-\delta_{ij}$.
\item\label{star_lem2}
For any effective $\Q$-divisor $\Delta$ on $X$, we have 
\[
K_{\tilde{X}}+\pi^{-1}_*\Delta-\pi^*(K_X+\Delta)\leq\frac{A_{(X, \Delta)}(F)-1}{f}E_m^*.
\]
\item\label{star_lem3}
For any $L\in\Pic X$, $j\in\Q_{\geq 0}$ and for any $k\in\Z_{>0}$ with 
$k/f$, $kj/f\in\Z$, the natural homomorphism 
\[
H^0\left(\tilde{X}, k\left(\pi^*L-\frac{j}{f}E_m^*\right)\right)
\to
H^0\left(\tilde{X}, k\left(\pi^*L-jF\right)\right)
\]
given by the effective divisor $kj((1/f)E_m^*-F)$ is an isomorphism. 
\item\label{star_lem4}
Assume that $F$ is primitive over $X$ and let $\sigma\colon Y\to X$ be the extraction 
of $F$. Then the natural morphism $\nu\colon\tilde{X}\to Y$ over $X$ is the 
minimal resolution of $Y$, and we have the equality
\[
\nu^*F=\frac{1}{f}E_m^*.
\]
\item\label{star_lem5}
Assume furthermore that $X$ is projective. Take an effective and nef $\Z$-divisor $P$ 
on $X$. Set 
\[
\pi^*P=:\tilde{P}+\sum_{i=1}^m n_i\tilde{E}_i,
\]
where $\tilde{P}:=\pi_*^{-1}P$. Take any $x\in[0, n_m]$ and set 
\[
M:=\pi^*P-\frac{x}{f}E_m^*=\nu^*\left(\sigma^*P-xF\right).
\]
If each irreducible component $P'$ of $\tilde{P}$ is nef, or more generally 
$(M\cdot P')\geq 0$, then 
$M$ is nef.
\end{enumerate}
\end{lemma}

\begin{proof}
\eqref{star_lem1}
Let us remark that $E_i^*=(\pi_{i+1}\circ\dots\circ\pi_m)^*(E_i^*)^{X_i}$ for any 
$1\leq i\leq m$. Thus it is obvious that $\left(E_1^*\cdot \tilde{E}_j\right)=-\delta_{1j}$. 
From now on, let us assume that $i\geq 2$. We may assume that 
$\left(E_{i'}^*\cdot \tilde{E}_j\right)=-\delta_{i'j}$ holds for any $1\leq i'<i$ and for any 
$1\leq j\leq m$. From the construction, we have the equality 
\[
\left(E_i^*\cdot \tilde{E}_j\right)=\left(\left(
E_{q(i)}^*+E_{i-1}^*+(\pi_{i+1}\circ\dots\circ\pi_m)^*E_i\right)\cdot \tilde{E}_j\right).
\]
If $j>i$, then we have $\left(E_i^*\cdot \tilde{E}_j\right)=0$ since $\tilde{E}_j$ is 
exceptional over $X_i$. 
If $j=i$, then we have 
\[
\left(E_i^*\cdot \tilde{E}_i\right)=\left(
(\pi_{i+1}\circ\dots\circ\pi_m)^*E_i\cdot \tilde{E}_i\right)=-1.
\]
If $j\in\{q(i)$, $i-1\}$, since 
$\left(\left(E_{q(i)}^*+E_{i-1}^*\right)\cdot \tilde{E}_j\right)=-1$ and 
\[
\left((\pi_{i+1}\circ\dots\circ\pi_m)^*E_i\cdot \tilde{E}_j\right)=1, 
\]
we have $\left(E_i^*\cdot \tilde{E}_j\right)=0$. 
If $j<i$ and $j\not\in\{q(i)$, $i-1\}$, then we have 
\[
\left(\left(E_{q(i)}^*+E_{i-1}^*\right)\cdot \tilde{E}_j\right)
=\left((\pi_{i+1}\circ\dots\circ\pi_m)^*E_i\cdot \tilde{E}_j\right)=0. 
\]
Thus we have $\left(E_i^*\cdot \tilde{E}_j\right)=0$. 

\eqref{star_lem2}
For any $1\leq i\leq m-1$, the self intersection number of $\tilde{E}_i$ is smaller than or 
equal to $-2$. In particular, we have $\left(K_{\tilde{X}}\cdot\tilde{E}_i\right)\geq 0$. 
Set
\[
\sum_{i=1}^{m-1}h_i\tilde{E}_i:=K_{\tilde{X}}+\pi^{-1}_*\Delta-\pi^*(K_X+\Delta)
-\frac{A_{(X, \Delta)}(F)-1}{f}E_m^*.
\]
By \eqref{star_lem1}, we have 
\[
\left(\sum_{i=1}^{m-1}h_i\tilde{E}_i\cdot\tilde{E}_j\right)
=\left(K_{\tilde{X}}+\pi^{-1}_*\Delta\cdot\tilde{E}_j\right)\geq
\left(K_{\tilde{X}}\cdot\tilde{E}_j\right)\geq 0
\]
for any $1\leq j\leq m-1$. Hence we have $h_i\leq 0$ for any $1\leq i\leq m-1$ 
by \cite[Lemma 3.41]{KoMo}.

\eqref{star_lem3}
Take any effective divisor $G$ on $X$ and set $g:=\coeff_F\pi^*G\in\Z_{\geq 0}$. 
Set 
\[
\sum_{i=1}^{m-1}g_i\tilde{E}_i:=\pi_*^{-1}G-\pi^*G+\frac{g}{f}E_m^*.
\]
By \eqref{star_lem1}, we have 
\[
\left(\sum_{i=1}^{m-1}g_i\tilde{E}_i\cdot\tilde{E}_j\right)
=\left(\pi_*^{-1}G\cdot\tilde{E}_j\right)\geq 0
\]
for any $1\leq j\leq m-1$. Again by \cite[Lemma 3.41]{KoMo}, we have $g_i\leq 0$ 
for any $1\leq i\leq m-1$. Thus we have $\pi^*G\geq (g/f)E_m^*$. The assertion 
follows from this fact. 

\eqref{star_lem4}
The set of $\nu$-exceptional curves is equal to $\{\tilde{E}_i\}_{1\leq i\leq m-1}$. 
Moreover, we have $\left(K_{\tilde{X}}\cdot\tilde{E}_i\right)\geq 0$ for any 
$1\leq i\leq m-1$. Thus $\nu$ is the minimal resolution of $Y$. Since 
$\left(E_m^*\cdot \tilde{E}_i\right)=0$ for any $1\leq i\leq m-1$, we have 
the equality $\nu^*F=(1/f)E_m^*$.

\eqref{star_lem5}
Since $\nu_*M=\nu_*\left(\tilde{P}+(n_m-x)\tilde{E}_m\right)$, the $\R$-divisor 
\[
M=\nu^*\nu_*\left(\tilde{P}+(n_m-x)\tilde{E}_m\right)
\]
is effective. 
Thus, if $P'$ is nef, then we get $(M\cdot P')\geq 0$. 
Obviously, we have $\left(M\cdot\tilde{E}_m\right)=x/f\geq 0$. From the 
assumption, for any irreducible curve $B$ on $Y$, we have 
$\left(M\cdot \nu_*^{-1}B\right)\geq 0$. Moreover, for any $\nu$-exceptional curve 
$G$ on $\tilde{X}$, we have $(M\cdot G)=0$. 
\end{proof}

\begin{definition}\label{plt_dfn}
Under the notations in Definitions \ref{monoidal_dfn} and \ref{star_dfn}, 
let us further assume that $F$ is plt-type over $X$. From the construction, $F$ on 
$\tilde{X}$ intersects $\tilde{E}_1\cup\cdots\cup\tilde{E}_{m-1}$ at most $2$ points. 
Moreover, by \cite[Theorem 4.15]{KoMo}, the dual graph of 
$\tilde{E}_1\cup\dots\cup\tilde{E}_m$ is a straight chain.
\begin{enumerate}
\renewcommand{\theenumi}{\arabic{enumi}}
\renewcommand{\labelenumi}{(\theenumi)}
\item\label{plt_dfn1}
Set 
$k:=\max\{2\leq i\leq m\,\,|\,\,q(i)=0\}$. 
From the structure of the dual graph of $\tilde{E}_1\cup\dots\cup\tilde{E}_m$, 
we have $q(i)=0$ for any $2\leq i\leq k$. 
\item\label{plt_dfn2}
For any $0\leq i\leq m$, let us define $a_i$, $b_i\in\Z_{\geq 0}$ as follows: 
\begin{itemize}
\item
$(a_0$, $b_0):=(1$, $0)$, $(a_1$, $b_1):=(1$, $1)$.
\item
$(a_i$, $b_i):=(a_{q(i)}$, $b_{q(i)})+(a_{i-1}$, $b_{i-1})$.
\end{itemize}
Moreover, let us set $a^F:=a_m$ and $b^F:=b_m$. Clearly, we have 
$a_i\geq b_i$ for any $0\leq i\leq m$. 
\end{enumerate}
\end{definition}

\begin{lemma}\label{plt_lem}
\begin{enumerate}
\renewcommand{\theenumi}{\arabic{enumi}}
\renewcommand{\labelenumi}{(\theenumi)}
\item\label{plt_lem1}
For any $1\leq i\leq m$, $a_i$ and $b_i$ are mutually prime. In particular, $a^F$ and 
$b^F$ are mutually prime. 
\item\label{plt_lem2}
For any $1\leq i\leq m$, we have $A_X(E_i)=a_i+b_i$. In particular, 
we have $A_X(F)=a^F+b^F$. 
\item\label{plt_lem3}
We have
$f=a^Fb^F$ and $k=\lceil(a^F/b^F)\rceil$, where $f$ be as in Definition \ref{monoidal_dfn}. 
\item\label{plt_lem4}
For any $1\leq i\leq k$, we have 
$\coeff_{\tilde{E}_i}E_m^*=\min\{ib^F, \, a^F\}$.
\end{enumerate}
\end{lemma}

\begin{proof}
All of the assertions are \'etale local. By \cite[Proposition 6.2.6]{prokhorov_MSJ}, 
we may assume that $X=\A^2$ and $\sigma\colon Y\to X$ is a toric morphism of toric 
varieties. Thus, there exist mutually prime $a$, $b\in\Z_{>0}$ with $a\geq b$ such that: 
\begin{itemize}
\item
$X$ corresponds to the fan $\Sigma_0$ in $N_\R$ (i.e., $X=X_{\Sigma_0}$) 
such that $\Sigma_0$ consists of the $2$-dimensional cone 
$\R_{\geq 0}(1, 0)+\R_{\geq 0}(0, 1)$ and all of its faces. 
\item
$Y$ corresponds to the fan $\Sigma'$ in $N_\R$ (i.e., $Y=X_{\Sigma'}$) 
such that $\Sigma'$ consists of the $2$-dimensional cones 
$\R_{\geq 0}(1, 0)+\R_{\geq 0}(a, b)$, $\R_{\geq 0}(a, b)+\R_{\geq 0}(0, 1)$, 
and all of those faces. 
\item
The morphism $\sigma\colon Y\to X$ corresponds to the natural morphism of fans. 
\end{itemize}
Let us consider the sequence of monoidal transforms with respects to $F$. 
Every step of the monoidal transform is a toric morphism. Let $\Sigma_i$ in $N_\R$ be 
the fan associates with $X_i$. Assume that $i<m$. Then $(a$, $b)\in N_\R$ belongs to 
the interior of some $2$-dimensional cone $\sigma_i\in\Sigma_i$. Moreover, $p_{i+1}$ 
is the torus-invariant point in $X_{i+1}$ corresponds to $\sigma_i$ and $\Sigma_{i+1}$ 
is obtained by the star subdivision of $\Sigma_i$ along $\sigma_i$ 
(see \cite[Definition 3.3.13]{CLS}). 

Let $(a^i$, $b^i)\in N$ be the primitive generator of the $1$-dimensional cone 
corresponds to $E_i\subset X_i$. Since $\sigma_i=\R_{\geq 0}(1, 0)+\R_{\geq 0}(i, 1)$ 
for any $0\leq i <a/b$, we have $(a^i$, $b^i)=(i$, $1)$ for any $1\leq i<a/b+1$.
Therefore, we have $q(i)=0$ for any $1\leq i\leq \lceil(a/b)\rceil$. Moreover, since 
\[
(a, b)\in\R_{\geq 0}(\lceil(a/b)\rceil, 1)+\R_{\geq 0}(1, 1),
\] 
we have $q(i)\neq 0$ 
for any $i>\lceil(a/b)\rceil$. Thus we have $k=\lceil(a/b)\rceil$. 
Moreover, from the construction, we have the equality 
\[
(a^i, b^i)=(a^{q(i)}, b^{q(i)})+(a^{i-1}, b^{i-1})
\]
for any $2\leq i\leq m$, where we set $(a^0$, $b^0):=(1$, $0)$. 
Hence we can inductively show that $(a_i$, $b_i)=(a^i$, $b^i)$ for any $1\leq i\leq m$. 
In particular, we have $(a^F$, $b^F)=(a$, $b)$.

\eqref{plt_lem1}
The assertion is trivial since $a^i$ and $b^i$ are mutually prime. 

\eqref{plt_lem2}
We know that 
\[
A_X(E_i)=\begin{cases}
A_X(E_{i-1})+1 & \text{if }q(i)=0, \\
A_X(E_{q(i)})+A_X(E_{i-1}) & \text{otherwise}.
\end{cases}
\]
Thus the assertion follows inductively. 

\eqref{plt_lem3}
We have already seen that $k=\lceil(a/b)\rceil=\lceil(a^F/b^F)\rceil$. 
Let $l_1$, $l_2\subset Y$ be the torus invariant curve on $Y$ corresponds to the 
$1$-dimensional cone $\R_{\geq 0}(1, 0)$, $\R_{\geq 0}(0, 1)$, respectively. 
Then we have $(F^{\cdot 2})_Y=-1/(ab)$ since $(F\cdot l_1)_Y=1/b$ and 
$(F\cdot l_2)_Y=1/a$ (see \cite[Theorem 15.1.1]{CLS}). Thus we have 
\[
-\frac{1}{a^Fb^F}=\left(\nu^*F\cdot\tilde{E}_m\right)=-\frac{1}{f}
\]
by Lemma \ref{star_lem}. 

\eqref{plt_lem4}
The dual graph of $\tilde{E}_1\cup\cdots\cup\tilde{E}_m$ on $\tilde{X}$ is of the form: 
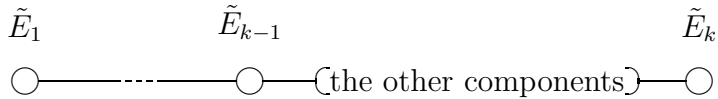
\begin{center}
    \begin{picture}(270, 55)(0, 50)
    \put(0, 60){\circle{10}}
    \put(0, 75){\makebox(0, 0)[b]{$\tilde{E}_1$}}
    \put(5, 60){\line(1, 0){31}}
    \put(38, 60){\line(1, 0){2}}
    \put(42, 60){\line(1, 0){2}}
    \put(46, 60){\line(1, 0){2}}
    \put(50, 60){\line(1, 0){30}}
    \put(85, 60){\circle{10}}
    \put(85, 75){\makebox(0, 0)[b]{$\tilde{E}_{k-1}$}}
    \put(90, 60){\line(1, 0){20}}
    \put(115, 60){\oval(10, 10)[l]}
    \put(171, 54){\makebox(0, 0)[b]{the other components}}
    \put(227, 60){\oval(10, 10)[r]}
    \put(232, 60){\line(1, 0){18}}
    \put(255, 60){\circle{10}}
    \put(255, 75){\makebox(0, 0)[b]{$\tilde{E}_k$}}
    \end{picture}
\end{center}
Moreover, $l_1^{\tilde{X}}$ (resp., $l_2^{\tilde{X}}$) intersects $\tilde{E}_k$ (resp., 
$\tilde{E}_1$) transversally. Thus we get 
\begin{eqnarray*}
\frac{1}{b^F}&=&(F\cdot l_1)_Y=\frac{1}{f}\left(E_m^*\cdot l_1^{\tilde{X}}\right)
=\frac{1}{a^Fb^F}\coeff_{\tilde{E}_k}E_m^*, \\
\frac{1}{a^F}&=&(F\cdot l_2)_Y=\frac{1}{f}\left(E_m^*\cdot l_2^{\tilde{X}}\right)
=\frac{1}{a^Fb^F}\coeff_{\tilde{E}_1}E_m^*.
\end{eqnarray*}
Thus the assertion is true when $i=1$ or $k$. We note that 
\begin{eqnarray*}
0&=&\left(E_m^*\cdot \tilde{E}_1\right)=-2\coeff_{\tilde{E}_1}E_m^*+
\coeff_{\tilde{E}_2}E_m^*, \\
0&=&\left(E_m^*\cdot \tilde{E}_{i-1}\right)=\coeff_{\tilde{E}_{i-2}}E_m^*
-2\coeff_{\tilde{E}_{i-1}}E_m^*+\coeff_{\tilde{E}_i}E_m^* 
\end{eqnarray*}
for $3\leq i\leq k-1$. 
Thus the assertion follows inductively. 
\end{proof}

\subsection{Dreamy prime divisors over log del Pezzo pairs}\label{dreamy_section}

We see basic properties of primes divisors over log del Pezzo pairs. 

\begin{proposition}\label{dream_prop}
Let $(X, \Delta)$ be a log del Pezzo pair and let $F$ be a prime divisor over $X$ with 
$\hat{\beta}_{(X, \Delta)}(F)<1/3$. Then $F$ is dreamy over $(X, \Delta)$. 
\end{proposition}

\begin{proof}
Set $L:=-(K_X+\Delta)$. 
By Proposition \ref{tau_prop}, we have $\tau(F)>A_{(X, \Delta)}(F)$. 
Let $\phi\colon X_0\to X$ be the minimal resolution of $X$ and set 
$K_{X_0}+\Delta_0:=\phi^*(K_X+\Delta)$. We know that $\Delta_0$ is effective. If $F$ 
is a prime divisor on $X_0$, since $-(K_{X_0}+\Delta_0)$ is nef and big, then $F$ is dreamy 
over $(X, \Delta)$ (see \cite[Corollary 1.3.2]{BCHM} for example). 
Assume that $F$ is exceptional over $X_0$. Let 
$\pi\colon\tilde{X}=X_m\to\dots\to X_1\to X_0$ be the sequence of monoidal transforms 
with respects to $F$. Since $\pi^*\phi^*L-A_{(X, \Delta)}(F)F$ is big, by Lemma 
\ref{star_lem} \eqref{star_lem3}, 
\[
\pi^*\phi^*L-\frac{A_{(X, \Delta)}(F)}{f}E_m^*
\] 
is also big. By Lemma \ref{star_lem} \eqref{star_lem2}, we have 
\[
\pi^*\phi^*L-\frac{A_{(X, \Delta)}(F)}{f}E_m^*\leq 
-\left(K_{\tilde{X}}+\pi^{-1}_*\Delta_0\right)\leq -K_{\tilde{X}}.
\]
This implies that $-K_{\tilde{X}}$ is big. Since $\tilde{X}$ is rational (see \cite{N} 
for example), the variety $\tilde{X}$ is a Mori dream space in the sense of \cite{HK} 
by \cite[Theorem 1]{TVAV}. Thus $F$ is dreamy over $(X, \Delta)$. 
\end{proof}

Let us recall the following: 

\begin{thm}\label{pltK_thm}
Let $(X, \Delta)$ be a log Fano pair and let $F$ be a primitive prime divisor over $X$. 
Let $\sigma\colon Y\to X$ be the extraction of $F$. 
Assume that $F$ is not plt-type over $(X, \Delta)$. Then there exists a plt-type 
prime divisor $G$ over $(X, \Delta)$ such that $c_X(G)\subset c_X(F)$ and 
$\hat{\beta}_{(X, \Delta)}(G)<\hat{\beta}_{(X, \Delta)}(F)$. 
\end{thm}

\begin{proof}
Follows directly from \cite[Theorem 3.1 and Corollary 3.2]{pltK}. 
\end{proof}

Finally, we give an example of non-dreamy prime divisor over $\pr^2$. 

\begin{example}\label{non-dreamy_ex}
Assume that $\Bbbk$ is uncountable. 
Set $X:=\pr^2$ and $L:=-K_X$. Fix any smooth cubic curve $B\subset X$. Take a very general point $p_1\in B$ with respects to the inflection points. Then we have 
$\sO_X(i)|_B\otimes\sO_B(-3ip_1)\not\simeq\sO_B$ for any $i\in\Z_{>0}$. 
Let us consider the sequence of monoidal transforms 
$\pi\colon\tilde{X}=X_9\to\cdots\to X_1\to X_0=X$
obtained by: 
\begin{itemize}
\item
$X_0:=X$, $\pi_1\colon X_1\to X_0$ is the blowup along $p_1$ and $E_1\subset X_1$ 
is the $\pi_1$-exceptional curve. 
\item
For $1\leq i\leq 8$, let $p_{i+1}\in X_i$ be the intersection of $E_i$ and $B^{X_i}$. 
$\pi_{i+1}\colon X_{i+1}\to X_{i}$ is the blowup along $p_{i+1}$ and $E_{i+1}\subset X_{i+1}$ 
is the $\pi_{i+1}$-exceptional curve. 
\end{itemize}
Let $\tilde{E}_i\subset\tilde{X}$ (resp., $\tilde{B}$) be the strict transform of $E_i$ 
(resp., $B$) on $\tilde{X}$ as in Definition \ref{monoidal_dfn}. 
Since $\tilde{E}_1,\dots,\tilde{E}_8$ are $(-2)$-curves and their dual graph is a straight 
chain, by \cite[Proposition 4.10]{KoMo}, 
the morphism $\pi\colon\tilde{X}\to X$ decomposes into 
\[
\tilde{X}\xrightarrow{\nu}Y\xrightarrow{\sigma}X
\]
such that the set of $\nu$-exceptional divisors is equal to the set 
$\{\tilde{E}_i\}_{1\leq i\leq 8}$. Set $F:=\nu_*E_9$. Obviously, the morphism $\sigma$ 
is the extraction of $F$ and $\pi$ is the sequence of monoidal transforms 
with respects to $F$. Moreover, by \cite[Theorem 4.15]{KoMo}, $F$ is plt-type over $X$. 
Since $E_9^*=\sum_{i=1}^9i\tilde{E}_i$, we have 
\[
\nu^*F=\sum_{i=1}^9\frac{i}{9}\tilde{E}_i
\]
by Lemma \ref{star_lem} \eqref{star_lem4}. Moreover, since 
$\tilde{B}=\pi^*B-\sum_{i=1}^9i\tilde{E}_i$, we have 
$\tilde{B}\sim_\Q\nu^*(\sigma^*L-9F)$.
Since $\tilde{B}$ is an irreducible curve and $\left(\tilde{B}^{\cdot 2}\right)=0$, the 
divisor $\sigma^*L-9F$ is nef and non-big. Thus we have $\varepsilon(F)=\tau(F)=9$. 

Assume that $F$ is dreamy over $X$. Then, as in Definition \ref{diag_dfn}, 
$\sigma^*L-9F$ is semiample. Thus, 
\[
\nu^*(\sigma^*L-9F)|_{\tilde{B}}=L|_B-9p_1\,\,(\equiv 0)
\]
is also semiample. This leads to a contradiction. Thus $F$ is non-dreamy over $X$. 
\end{example}

\section{Product-type prime divisors}\label{product_section}

\subsection{Over $\mathbb{P}^2$}\label{product_P_section}

In this section, let $C\subset \pr^2$ be a smooth conic, and let us take 
$\delta\in[0,1)\cap\Q$ and set $\Delta:=\delta C$. 
It is well-known that, any $1$-parameter subgroup of $\Aut(\pr^2)=\PGL(3)$ is, 
after a coordinate change of $\pr^2$, of the form 
\begin{eqnarray*}
\rho\colon\G_m&\to&\PGL(3)\\
t&\mapsto&\diag(1, t^{-a'}, t^{-b'})
\end{eqnarray*}
for some $(a'$, $b')\in\Z^{2}_{\geq 0}\setminus\{(0$, $0)\}$ with $a'\geq b'$. 
Let $g\in\Z_{>0}$ be the greatest common factor of $a'$ and $b'$, and let us set 
$a:=a'/g$, $b:=b'/g$. As in \cite[Example 3.6]{ha} (see also \cite{JM}), the divisorial 
valuation $v$ on $\Bbbk(\pr^2)$ associates to $\rho$ is the quasi-monomial valuation 
on 
\[
\A^2_{x_1, x_2}=\pr^2_{z_0:z_1:z_2}\setminus(z_0=0)
\]
for coordinates $(x_1$, $x_2)$ (where $x_i:=z_i/z_0$) with weights $(a'$, $b')$. 
If $b'=0$, then $v=a'\cdot\ord_l$, where $l\subset \pr^2$ is the line $(z_1=0)$. 
Assume that $b'\geq 1$. Let $\Sigma^{(a,b)}$ (resp., $\Sigma$) be the complete fan 
in $N_\R$ such that the set of $1$-dimensional cones is equal to 
\begin{eqnarray*}
\{\R_{\geq 0}(1,0), \,\,\R_{\geq 0}(a,b), \,\,\R_{\geq 0}(0,1), \,\,\R_{\geq 0}(-1,-1)\}\\
(\text{resp., }\,\{\R_{\geq 0}(1,0), \,\,\R_{\geq 0}(0,1), \,\,\R_{\geq 0}(-1,-1)\}).
\end{eqnarray*}
Set $Y^{(a,b)}:=X_{\Sigma^{(a,b)}}$ and let $\sigma\colon Y^{(a,b)}\to X_{\Sigma}=\pr^2$ 
be the natural toric morphism. 
Let $F^{(a,b)}\subset Y^{(a,b)}$ be the $\sigma$-exceptional divisor. Then we know that 
$v=g\cdot\ord_{F^{(a,b)}}$. 

Consequently, we have proved the following: 

\begin{lemma}\label{prod1_lem}
A prime divisor $F$ over $\pr^2$ is product-type over $\pr^2$ if and only if $F$ is a 
line on $\pr^2$ or, after a coordinate change of $\pr^2$, $F$ is equal to the above 
$F^{(a,b)}$ for some $a$, $b\in\Z_{>0}$ with $a$, $b$ mutually prime. 
\end{lemma}

We consider product-type prime divisor over $(\pr^2, \Delta)$. 
Take any point $p_1\in C$. Let $\pi_1\colon X_1\to \pr^2$ be the blowup along $p_1$ and 
let $E_1\subset X_1$ be the $\pi_1$-exceptional curve. Let $p_2\in X_1$ be the 
intersection of $E_1$ and $C^{X_1}$. Let $\pi_2\colon X_2\to X_1$ be the blowup along 
$p_2$ and let $E_2\subset X_2$ be the $\pi_2$-exceptional curve. 

\begin{proposition}\label{prod2_prop}
The above $E_2$ is a product-type prime divisor over $(\pr^2, \Delta)$. 
\end{proposition}

\begin{proof}
We may assume that $p_1=(1:0:0)$ and $C=(z_2^2+z_0z_1=0)\subset\pr^2_{z_0:z_1:z_2}$. 
Then, as we have seen in Lemma \ref{prod1_lem}, the divisorial valuation 
$\ord_{E_2}\colon\Bbbk(\pr^2)^*\to\Z$ corresponds to the $1$-parameter subgroup
\begin{eqnarray*}
\rho\colon\G_m&\to&\PGL(3)\\
t&\mapsto&\diag(1, t^{-2}, t^{-1}).
\end{eqnarray*}
For any $t\in\G_m$, we have $\rho_t^*C=C$. Thus $\rho$ factors through 
\[
\Aut(\pr^2,\Delta)\subset\Aut(\pr^2)=\PGL(3). 
\]
Hence $E_2$ is a product-type prime divisor over $(\pr^2, \Delta)$. 
\end{proof}

\subsection{Over $\mathbb{P}^1\times\mathbb{P}^1$}\label{product_PP_section}

In this section, let $C\subset \pr^1\times\pr^1$ be the diagonal, and let 
$\delta\in[0$, $1)\cap\Q$, and set $\Delta:=\delta C$. Let us consider the 
1-parameter subgroup
\begin{eqnarray*}
\rho\colon\G_m & \to & \PGL(2)\times\PGL(2) \subset\Aut(\pr^1\times\pr^1)\\
t & \mapsto & \left(\diag(1, t^{-1}), \diag(1, t^{-1})\right). 
\end{eqnarray*}
Since $C\subset\pr^1_{z_{10}:z_{11}}\times\pr^1_{z_{20}:z_{21}}$ is defined by the 
equation $z_{10}z_{21}=z_{11}z_{20}$, we have $\rho_t^*C=C$ for any $t\in\G_m$. 
Thus $\rho$ factors through $\Aut(\pr^1\times\pr^1, \Delta)$. 
On the other hand, the morphism 
\begin{eqnarray*}
\rho\colon \G_m\times\pr^1\times\pr^1 & \to & \pr^1\times\pr^1\\
\left(t; z_{10}: z_{11}; z_{20}: z_{21}\right) & \mapsto & 
\left(z_{10}: t^{-1}z_{11}; z_{20}: t^{-1}z_{21}\right),
\end{eqnarray*}
induces the inclusion 
\begin{eqnarray*}
\rho^*\colon \Bbbk(x_1, x_2) & \to & \Bbbk(x_1,x_2)(t) \\
x_1 & \mapsto & t^{-1}x_1, \\
x_2 & \mapsto & t^{-1}x_2, 
\end{eqnarray*}
where $x_i:=z_{i1}/z_{i0}$ for $i=1$, $2$. Thus, as in \cite[\S 3]{ha}, the divisorial valuation 
$v$ on $\Bbbk(\pr^1\times\pr^1)$ associates to $\rho$ is the quasi-monomial valuation 
on $\A^2_{x_1,x_2}=\pr^1\times\pr^1\setminus (z_{10}z_{20}=0)$ for coordinates 
$(x_1$, $x_2)$ with weights $(1$, $1)$. In other words, if $F$ is the exceptional divisor 
of the ordinary blowup of $\pr^1\times\pr^1$ along $(1:0 ; 1:0)$, then $v$ is equal to 
$\ord_F$. Thus we have proved the following proposition: 

\begin{proposition}\label{PP_prod_prop}
Let $F$ be the exceptional divisor of the ordinary blowup of 
$\pr^1\times\pr^1$ along a point on $C$. Then $F$ is product-type over 
$(\pr^1\times\pr^1, \Delta)$. 
\end{proposition}

In order to prove Theorem \ref{mainthm} \eqref{mainthm2}, we need the following lemma: 

\begin{lemma}\label{PP_non_prod_lem}
The divisor $C$ on $\pr^1\times\pr^1$ is not a product-type prime divisor over 
$(\pr^1\times\pr^1, \Delta)$. 
\end{lemma}

\begin{proof}
Assume not. Then, as in the proof in \cite[Lemma 3.8]{ha}, 
$\pr^1\times\pr^1$ must be isomorphic to 
\[
\sX_0:=\Proj\bigoplus_{k\in\Z_{\geq 0}}\left(\bigoplus_{j\in\Z_{\geq 0}}S_{k,j}\right), 
\]
where $S_{k,j}:=\sF^jV_k/\sF^{j+1}V_k$ with 
\[
\sF^jV_k:=H^0\left(\pr^1\times\pr^1, \sO_{\pr^1\times\pr^1}(k,k)(-jC)\right). 
\]
Note that 
\[
\sF^jV_k=\biggl\{(x_1-x_2)^j g(x_1,x_2)\,\,\bigg|\,\, g(x_1,x_2)\in\Bbbk[x_1,x_2]; \,\,
\begin{matrix}
\deg_{x_1}g \leq k-j,\,\\
\deg_{x_2}g \leq k-j.
\end{matrix}\biggr\}.\]
Thus we have the natural isomorphism 
\begin{eqnarray*}
S_{k,j} & \to & \{h(x)\in\Bbbk[x]\,|\, \deg h\leq 2k-2j\}\\
g(x_1, x_2) \,\,\MOD \,\, \sF^{j+1}V_k& \mapsto & g(x,x).
\end{eqnarray*}
In particular, the variety 
$\sX_0$ is isomorphic to the weighted projective plane $\pr(1,1,2)$ of 
weights $(1,1,2)$. This leads to a contradiction. 
\end{proof}

\section{On the projective plane}\label{P_section}

In this section, we set $X:=\pr^2$, let $C\subset X$ be a smooth conic, 
fix $\delta\in[0,1)\cap\Q$ and set $\Delta:=\delta C$. We prove the following theorem. 
Theorem \ref{mainthm} \eqref{mainthm1} is an immediate consequence 
of Theorem \ref{F_thm}. 

\begin{thm}\label{F_thm}
\begin{enumerate}
\renewcommand{\theenumi}{\arabic{enumi}}
\renewcommand{\labelenumi}{(\theenumi)}
\item\label{F_thm1}
If $\delta>3/4$ $($resp., $\delta\geq 3/4)$, then $(X, \Delta)$ is not K-semistable 
$($resp., not K-polystable$)$. 
\item\label{F_thm2}
The pair $(X, \Delta)$ is no longer K-stable for any $\delta\in[0,1)\cap\Q$. 
\item\label{F_thm3}
Assume that $\delta\leq 3/4$. 
For any prime divisor $F$ over $X$, we have $\hat{\beta}_{(X, \Delta)}(F)\geq 0$. 
\item\label{F_thm4}
If $\delta<3/4$ and if a prime divisor $F$ over $X$ satisfies that 
$\hat{\beta}_{(X, \Delta)}(F)=0$, 
then $F$ is a product-type prime divisor over $(X, \Delta)$. 
\end{enumerate}
\end{thm}

\begin{proof}
The proof is based on the ideas in \cite[\S 4.2]{pltK}. However, we need 
more delicate arguments. 

\textbf{Step 1.}
Take any prime divisor $F$ \emph{on} $X$. Set $d:=\deg F$. If $F\neq C$, then we have 
\[
\hat{\beta}_{(X, \Delta)}(F)=
1-\frac{\int_0^{\frac{3-2\delta}{d}}(3-2\delta-d\cdot x)^2dx}{(3-2\delta)^2}
=1-\frac{3-2\delta}{3d}\geq 1-\frac{1}{d}\geq 0. 
\]
Moreover, equality holds if and only if $d=1$ and $\delta=0$. We already know in 
Lemma \ref{prod1_lem} that a line is product-type prime divisor over $X$. 
If $F=C$, then $d=2$ and $A_{(X, \Delta)}(C)=1-\delta$. Thus we have 
\[
\hat{\beta}_{(X, \Delta)}(C)=\frac{3-4\delta}{6(1-\delta)}. 
\]
By Lemma \ref{prod1_lem}, $C$ is not a product-type prime divisor over $(X, \Delta)$. 
Thus we have proved Theorem \ref{F_thm} \eqref{F_thm1}. 

\textbf{Step 2.}
Let us prove Theorem \ref{F_thm} \eqref{F_thm2}, \eqref{F_thm3} and \eqref{F_thm4}. 
From now on, we assume that $\delta\leq 3/4$. 
Let $F$ be a prime divisor over $X$. 
By Step 1, Proposition \ref{dream_prop} and Theorem \ref{pltK_thm}, we may assume that 
$F$ is exceptional over $X$, dreamy over $(X, \Delta)$ and plt-type over $(X, \Delta)$. 
Of course, $F$ is plt-type over $X$. 
Let $\pi\colon\tilde{X}=X_m\to\dots\to X_1\to X_0=X$, $E_i^*$ $(1\leq i\leq m)$, 
$a^F$, $b^F$, $f$, etc., be as in Definitions \ref{monoidal_dfn}, \ref{star_dfn} and 
\ref{plt_dfn}. Moreover, let us set 
$a:=a^F$, $b:=b^F$, $\varepsilon:=\varepsilon(F)$, $\tau:=\tau(F)$, 
$A_i:=A_{(X, \Delta)}(E_i)$ $(1\leq i\leq m)$ and $A:=A_{(X, \Delta)}(F)$ for simplicity. 

Let $l_1\subset X$ be a general line passing through $p_1$. By Lemmas \ref{hinpan_lem}, 
\ref{star_lem} and \ref{plt_lem}, we have 
\[
0\leq\left(\left(\pi^*L-\frac{\tau}{ab}E_m^*\right)\cdot l_1^{\tilde{X}}\right)
=3-2\delta-\frac{\tau}{a}. 
\]
Thus we get $\tau\leq a(3-2\delta)$ and $\varepsilon+\tau\leq(a+b)(3-2\delta)$ 
(recall that $\varepsilon\tau=ab(3-2\delta)^2$ by Lemma \ref{beta_lem}). 

\textbf{Step 3.}
We consider the case $m=1$, i.e., $a=b=1$. Then, since 
$Y=\pr_{\pr^1}(\sO\oplus\sO(1))$, we have $\varepsilon=\tau=3-2\delta$. 
If $p_1\not\in C$, then we have $A=2$; if $p_1\in C$, then we have $A=2-\delta$. 
By Lemma \ref{beta_lem}, we get 
\[
\hat{\beta}_{(X, \Delta)}(F)=\begin{cases}
\frac{2}{3}\delta & \text{if }p_1\not\in C,\\
\frac{\delta}{6-3\delta} & \text{if } p_1\in C.
\end{cases}\]
Hence we have $\hat{\beta}_{(X, \Delta)}(F)\geq 0$. 
If $\hat{\beta}_{(X, \Delta)}(F)=0$, then $\delta=0$ and 
$F$ is a product-type prime divisor over $(X, \Delta)\,(=X)$ by Lemma \ref{prod1_lem}. 

\textbf{Step 4.}
Thus we may further assume that $m\geq 2$. 
Let $l_0\subset X$ be the unique line such that $l_0^{X_2}\cap E_2\neq\emptyset$. 
Let us set 
\[
j_0:=\max\{2\leq i\leq k\,\,|\,\,l_0^{X_i}\cap E_i\neq\emptyset\}, 
\]
where $k=\lceil(a/b)\rceil$ as in Definition \ref{plt_dfn}. 

We consider the case $p_2\not\in C^{X_1}$. If $p_1\in C$ and $\delta>0$, then 
we have $A_1=2-\delta$ and 
\[
A_i=\begin{cases}
A_{i-1}+1 & 2\leq i\leq k, \\
A_{q(i)}+A_{i-1} & k+1\leq i\leq m.
\end{cases}\]
Thus we can inductively show that $A_i=a_i+b_i-b_i\delta$. In particular, we have 
$A=a+b-b\delta$. By Lemma \ref{beta_lem} and Step 2, we have 
\[
\hat{\beta}_{(X, \Delta)}(F)\geq 1-\frac{(a+b)(3-2\delta)}{3(a+(1-\delta)b)}
=\frac{\delta(2a-b)}{3(a+(1-\delta)b)}>0.
\]
If $p_1\not\in C$, then we have $A=a+b$. By Lemma \ref{beta_lem} and Step 2, we have 
\[
\hat{\beta}_{(X, \Delta)}(F)\geq 1-\frac{(a+b)(3-2\delta)}{3(a+b)}=\frac{2}{3}\delta\geq 0
\]
by Step 2. Assume that $\hat{\beta}_{(X, \Delta)}(F)=0$. Then $\delta=0$ and 
$(\varepsilon$, $\tau)=(3b$, $3a)$. Thus $\mu\colon Y\to Z$ is birational, where $\mu$ 
is as in Definition \ref{diag_dfn}. If $l_0^Y$ is not $\mu$-exceptional, then, 
by Lemma \ref{hinpan_lem}, we have 
\[
0\leq\left(\left(\pi^*L-\frac{\tau}{ab}E_m^*\right)\cdot l_0^{\tilde{X}}\right)
=3\left(1-\frac{1}{b}\coeff_{\tilde{E}_{j_0}}E_m^*\right). 
\]
However, by Lemma \ref{plt_lem}, we have $\coeff_{\tilde{E}_{j_0}}E_m^*>b$. This leads to 
a contradiction. Thus $l_0^Y$ is $\mu$-exceptional. 
By Lemma \ref{hinpan_lem}, we have 
\[
0=\left(\left(\pi^*L-\frac{\varepsilon}{ab}E_m^*\right)\cdot l_0^{\tilde{X}}\right)
=3\left(1-\frac{1}{a}\coeff_{\tilde{E}_{j_0}}E_m^*\right). 
\]
Thus we have $\coeff_{\tilde{E}_{j_0}}E_m^*=a$. This implies that $j_0=k$. 
As we have already seen in Lemma \ref{prod1_lem}, $F$ is a product-type prime 
divisor over $(X, \Delta)\,(=X)$. 

\textbf{Step 5.}
Thus we may further assume that $\delta>0$ and $p_2\in C^{X_1}$. 
In this case, $l_0^{X_1}$ and $C^{X_1}$ intersect transversally at $p_2$. 

We consider the case $m=2$. In this case, $F$ is product-type over $(X, \Delta)$ 
by Proposition \ref{prod2_prop}. Since $\tilde{X}$ is toric, we can easily show that 
$l_0^Y$ is the unique $\mu$-exceptional curve. By Lemma \ref{hinpan_lem}, we have 
\[
0=\left(\left(\pi^*L-\frac{\varepsilon}{2}E_2^*\right)\cdot l_0^{\tilde{X}}\right)
=3-2\delta-\varepsilon.
\]
Thus we have $(\varepsilon$, $\tau)=(3-2\delta$, $2(3-2\delta))$, and 
$\hat{\beta}_{(X, \Delta)}(F)=0$ by Lemma \ref{beta_lem} (note that $A=3-2\delta$). 
In particular, we have proved Theorem \ref{F_thm} \eqref{F_thm2}. 

\textbf{Step 6.}
Thus we may further assume that $m\geq 3$. We consider the case $p_3\in E_1^{X_2}$. 
In this case, we have $k=2$. Thus $2b>a$ holds. Since $A_1=2-\delta$, 
$A_2=3-2\delta$ and $A_i=A_{q(i)}+A_{i-1}$ for any $3\leq i\leq m$, we can inductively 
show that $A_i=a_i+b_i-a_i\delta$. In particular, we have $A=a+b-a\delta$. 
By Lemma \ref{beta_lem} and Step 2, we get 
\[
\hat{\beta}_{(X, \Delta)}(F)\geq 1-\frac{(a+b)(3-2\delta)}{3((1-\delta)a+b)}
=\frac{\delta(2b-a)}{3((1-\delta)a+b)}>0.
\]

\textbf{Step 7.}
Thus we may further assume that $p_3\not\in E_1^{X_2}$. Then $k\geq 3$. In particular, 
we have $a>2b$. Let us set 
\[
j_C:=\max\{2\leq i\leq k\,\,|\,\,E_i\cap C^{X_i}\neq \emptyset\}.
\]
Since $l_0^{X_2}\cap C^{X_2}=\emptyset$, either $j_0$ or $j_C$ is equal to $2$. 
By the definitions of $j_C$ and $k$, we have 
\[
A_i=\begin{cases}
i+1-i\delta & 2\leq i\leq j_C, \\
i+1-j_C\delta & j_C\leq i\leq k, \\
A_{q(i)}+A_{i-1} & k+1\leq i\leq m.
\end{cases}\]
Therefore, we can inductively show that, for any $k\leq i\leq m$, 
\[
A_i=\begin{cases}
a_i+b_i-j_Cb_i\delta & \text{if }j_C<k, \\
a_i+b_i-a_i\delta & \text{if }j_C=k.
\end{cases}\]
In particular, we have $A=a+b-\min\{j_Cb$, $a\}\delta$. 

Assume that $j_C=2$. Then $A=a+b-2b\delta$.
By Lemma \ref{beta_lem} and Step 2, we get 
\[
\hat{\beta}_{(X, \Delta)}(F)\geq 1-\frac{(a+b)(3-2\delta)}{3(a+(1-2\delta)b)}
=\frac{2\delta(a-2b)}{3(a+(1-2\delta)b)}>0.
\]

\textbf{Step 8.}
Thus we may further assume that $j_C\geq 3$. This implies that $j_0=2$. 
If $l_0^Y$ is $\mu$-exceptional, then we have 
\[
0=\left(\left(\pi^*L-\frac{\varepsilon}{ab}E_m^*\right)\cdot l_0^{\tilde{X}}\right)
=3-2\delta-\frac{2\varepsilon}{a}
\]
by Lemma \ref{hinpan_lem}. 
Thus we get $\varepsilon+\tau=(a/2+2b)(3-2\delta)$. 
If $l_0^Y$ is not $\mu$-exceptional, then we have 
\[
0\leq\left(\left(\pi^*L-\frac{\tau}{ab}E_m^*\right)\cdot l_0^{\tilde{X}}\right)
=3-2\delta-\frac{2\tau}{a}.
\]
by Lemma \ref{hinpan_lem}. Thus, in any case, we have the inequality 
\[
\varepsilon+\tau\leq\left(\frac{a}{2}+2b\right)(3-2\delta). 
\]
By Lemma \ref{beta_lem}, we have 
\begin{eqnarray*}
\hat{\beta}_{(X, \Delta)}(F)&\geq&1-
\frac{(\frac{a}{2}+2b)(3-2\delta)}{3(a+b-\min\{j_Cb,a\}\delta)}\\
&\geq&1-\frac{(a+4b)(3-2\delta)}{6((1-\delta)a+b)}
=\frac{(3-4\delta)(a-2b)}{6((1-\delta)a+b)}\geq 0.
\end{eqnarray*}
Moreover, if $\delta<3/4$, then $\hat{\beta}_{(X, \Delta)}(F)>0$. 

As a consequence, we have completed the proof of Theorem \ref{F_thm}. 
\end{proof}

\begin{remark}\label{P2_rmk}
One may expects that there might be a positive constant $\varepsilon_0$ such that 
$\hat{\beta}_X(F)\geq \varepsilon_0$ holds for any non-product-type 
prime divisor 
$F$ over $X=\pr^2$. However, this is not true. See the following example. 
\end{remark}

\begin{example}\label{P2_ex}
Let $l\subset X$ be a line. Fix any $m\geq 4$. Take any point $p_1\in l$ and let us 
consider the sequence of monoidal transforms 
$\pi'\colon X_{m-1}\to\cdots\to X_1\to X_0$ obtained by: 
\begin{itemize}
\item
$X_0:=X$, $\pi_1\colon X_1\to X_0$ is the blowup along $p_1$ and let $E_1\subset X_1$ 
be the $\pi_1$-exceptional curve. 
\item
For any $2\leq i\leq m-2$, let $p_{i+1}\in X_i$ be the intersection of $l^{X_i}$ and $E_i$, 
let $\pi_{i+1}\colon X_{i+1}\to X_i$ be the blowup along $p_{i+1}$, and let 
$E_{i+1}\subset X_{i+1}$ be the $\pi_{i+1}$-exceptional curve.
\end{itemize}
Moreover, let us take $p_m\in E_{m-1}$ with 
$p_m\not\in l^{X_{m-1}}\cup E_{m-2}^{X_{m-1}}$, let $\pi_m\colon X_m\to X_{m-1}$ 
be the blowup along $p_m$ and let $E_m\subset X_m$ be the $\pi_m$-exceptional 
curve. Set $\pi:=\pi'\circ\pi_m$, $\tilde{X}:=X_m$ and let $\tilde{E}_i$ (resp., $\tilde{l}$) 
be the strict transform of $E_i$ (resp., $l$) on $\tilde{X}$. 
Then the dual graph of $\tilde{E}_1,\dots,\tilde{E}_m, \tilde{l}$ is the following: 
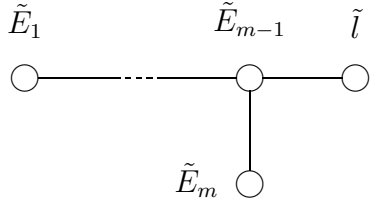
\begin{center}
    \begin{picture}(140, 75)(0, 20)
    \put(0, 60){\circle{10}}
    \put(0, 75){\makebox(0, 0)[b]{$\tilde{E}_1$}}
    \put(5, 60){\line(1, 0){31}}
    \put(38, 60){\line(1, 0){2}}
    \put(42, 60){\line(1, 0){2}}
    \put(46, 60){\line(1, 0){2}}
    \put(50, 60){\line(1, 0){30}}
    \put(85, 60){\circle{10}}
    \put(85, 75){\makebox(0, 0)[b]{$\tilde{E}_{m-1}$}}
    \put(90, 60){\line(1, 0){30}}
    \put(125, 60){\circle{10}}
    \put(125, 75){\makebox(0, 0)[b]{$\tilde{l}$}}
    \put(85, 55){\line(0, -1){30}}
    \put(85, 20){\circle{10}}
    \put(65, 15){\makebox(0, 0)[b]{$\tilde{E}_m$}}
    \end{picture}
\end{center}
Note that $\left(\tilde{E}_i^{\cdot 2}\right)=-2$ for $1\leq i\leq m-1$ and 
$\left(\tilde{l}^{\cdot 2}\right)=-(m-2)$. 
By \cite[Proposition 4.10]{KoMo}, the morphism $\pi$ decomposes into 
\[
\tilde{X}\xrightarrow{\nu}Y\xrightarrow{\sigma}X
\]
such that the set of $\nu$-exceptional divisors on $\tilde{X}$ is equal to the set 
$\{\tilde{E}_i\}_{1\leq i\leq m-1}$. Obviously, $F:=\nu_*\tilde{E}_m$ is primitive over $X$ 
and $F$ is plt-type over $X$ by \cite[Theorem 4.15]{KoMo}. 
Again by \cite[Proposition 4.10]{KoMo}, we can contract 
$\tilde{E}_1,\dots,\tilde{E}_{m-1}$, 
$\tilde{l}$. In particular, there exists a birational morphism $\mu\colon Y\to Z$ such that 
$l^Y$ is the unique $\mu$-exceptional curve. This implies that $F$ is dreamy over $X$ 
(see \cite{HK} for example) and the standard diagram with respects to $F$ consists of 
$\sigma$ and $\mu$. 

From the construction, we have $a^F=m$, $b^F=1$, $f=m$, $A_X(F)=m+1$. 
Moreover, by Lemma \ref{hinpan_lem}, we have the equality
\[
0=\left(\left(\pi^*L-\frac{\varepsilon(F)}{m}E_m^*\right)\cdot \tilde{l}\right)
=3-\frac{\varepsilon(F)}{m}\coeff_{\tilde{E}_{m-1}}E_m^*.
\]
By Lemma \ref{star_lem}, we have $\varepsilon(F)=3m/(m-1)$. 
By Lemma \ref{beta_lem}, we get 
\[
\hat{\beta}_X(F)
=1-\frac{\frac{3m}{m-1}+3(m-1)}{3(m+1)}=\frac{m-2}{(m+1)(m-1)}.
\]
Therefore, we have $\hat{\beta}_X(F)\searrow 0$ when $m\to \infty$. 
\end{example}

\section{On the product of the projective lines}\label{PP_section}

In this section, we set $X:=\pr^1\times\pr^1$, let $C\subset X$ be the diagonal, 
fix $\delta\in(0,1)\cap\Q$ and set $\Delta:=\delta C$. We recall the 
following result: 

\begin{thm}[{see \cite{kempf, li, blum, BJ}}]\label{PPKss_thm}
$X=\pr^1\times\pr^1$ is K-semistable. 
\end{thm}

\begin{remark}\label{PPKss_rmk}
When $\Bbbk=\C$, the above result is well-known (see \cite{tian, don}). We emphasize 
that some proofs of Theorem \ref{PPKss_thm} are purely algebraic. When $\Bbbk=\C$, 
the K-polystability of $X$ is also known (see \cite{B}). Moreover, recently, 
K-polystability of $X$ was proved purely algebraically by \cite{LWX}. 
\end{remark}

In this section, we algebraically prove the following theorem 
by using Theorem \ref{PPKss_thm}. Theorem \ref{mainthm} \eqref{mainthm2} is an 
immediate consequence of Theorem \ref{FF_thm}.

\begin{thm}\label{FF_thm}
\begin{enumerate}
\renewcommand{\theenumi}{\arabic{enumi}}
\renewcommand{\labelenumi}{(\theenumi)}
\item\label{FF_thm1}
If $\delta>1/2$ $($resp., $\delta\geq 1/2)$, then $(X, \Delta)$ is not K-semistable 
$($resp., not K-polystable$)$. 
\item\label{FF_thm2}
The pair $(X, \Delta)$ is no longer K-stable for any $\delta\in(0,1)\cap\Q$. 
\item\label{FF_thm3}
Assume that $\delta\leq 1/2$. 
For any prime divisor $F$ over $X$, we have $\hat{\beta}_{(X, \Delta)}(F)\geq 0$. 
\item\label{FF_thm4}
If $\delta<1/2$ and if a prime divisor $F$ over $X$ satisfies that 
$\hat{\beta}_{(X, \Delta)}(F)=0$, 
then $F$ is a product-type prime divisor over $(X, \Delta)$. 
\end{enumerate}
\end{thm}

\begin{proof}
The proof is similar to the proof of Theorem \ref{F_thm}. The proof of 
Theorem \ref{FF_thm} is more complicated than the proof of Theorem \ref{F_thm}. 

\textbf{Step 1.}
Since $A_{(X, \Delta)}(C)=1-\delta$, we have 
\[
\hat{\beta}_{(X, \Delta)}(C)=1-\frac{1}{1-\delta}\cdot\frac{\int_0^{2-\delta}2(2-\delta-x)^2
dx}{2(2-\delta)^2}=\frac{1- 2\delta}{3(1-\delta)}.
\]
By Lemma \ref{PP_non_prod_lem}, we have proved Theorem \ref{FF_thm} \eqref{FF_thm1}. 
We may assume that $\delta\leq 1/2$. 

\textbf{Step 2.}
Take any prime divisor $F$ over $X$. By Theorem \ref{PPKss_thm}, we have 
\begin{eqnarray*}
\hat{\beta}_{(X, \Delta)}(F)&=&1-\frac{1}{A_{(X, \Delta)}(F)}\cdot\frac{2-\delta}{2}\cdot
\frac{\int_0^\infty\vol(-K_X-xF)dx}{\left((-K_X)^{\cdot 2}\right)}\\
&=&
1-\frac{2-\delta}{2}\cdot\frac{A_X(F)
\left(1-\hat{\beta}_X(F)\right)}{A_{(X, \Delta)}(F)}\\
&\geq& 1-\frac{2-\delta}{2}\cdot\frac{A_X(F)}{A_{(X, \Delta)}(F)}.
\end{eqnarray*}
Thus, if $A_X(F)=A_{(X, \Delta)}(F)$, i.e., if $c_X(F)\not\subset C$ holds, 
then we have the inequality $\hat{\beta}_{(X, \Delta)}(F)\geq \delta/2>0$. 

\textbf{Step 3.}
Thus we may assume that $F$ is exceptional over $X$ and $c_X(F)\in C$. 
Moreover, by Proposition \ref{dream_prop} and Theorem \ref{pltK_thm}, 
we may assume that $F$ is dreamy over $(X, \Delta)$ and plt-type over $(X, \Delta)$. 
Let $l_1$, $l_2\subset X$ be the fibers of the fibrations 
$\pr^1_{z_{10}:z_{11}}\times\pr^1_{z_{20}:z_{21}}\to\pr^1_{z_{10}:z_{11}}$ and 
$\pr^1_{z_{10}:z_{11}}\times\pr^1_{z_{20}:z_{21}}\to\pr^1_{z_{20}:z_{21}}$ passing through 
$c_X(F)$. 
Let $\pi\colon\tilde{X}=X_m\to\dots\to X_1\to X_0=X$, $E_i^*$ $(1\leq i\leq m)$, 
$a^F$, $b^F$, $f$, etc., be as in Definitions \ref{monoidal_dfn}, \ref{star_dfn} and 
\ref{plt_dfn}. Moreover, let us set 
$a:=a^F$, $b:=b^F$, $\varepsilon:=\varepsilon(F)$, 
$A_i:=A_{(X, \Delta)}(E_i)$ $(1\leq i\leq m)$ and $A:=A_{(X, \Delta)}(F)$ for simplicity. 
For any $x\in[0$, $\varepsilon]$, we have 
\[
\vol(L-xF)=2(2-\delta)^2-\frac{x^2}{ab}. 
\]
From Proposition \ref{convex_prop}, we have 
\begin{eqnarray*}
&&\frac{1}{(L^{\cdot 2})}\int_0^\infty\vol(L-xF)dx\\
&\leq&\frac{1}{2(2-\delta)^2}\Biggl(\int_0^\varepsilon\left(2(2-\delta)^2-\frac{x^2}{ab}\right)dx\\
&&+\int_\varepsilon^{\varepsilon+\frac{2ab(2-\delta)^2
-\varepsilon^2}{\varepsilon}}\left(2(2-\delta)^2
-\frac{\varepsilon^2}{ab}\right)
\left(\frac{-\varepsilon(x-\varepsilon)}{2ab(2-\delta)^2-\varepsilon^2}
+1\right)^2dx\Biggr)\\
&=&\frac{1}{3\varepsilon}\left(2ab(2-\delta)^2+\varepsilon^2\right). 
\end{eqnarray*}
Therefore we get the inequality 
\[
\hat{\beta}_{(X, \Delta)}(F)\geq 1-\frac{2ab(2-\delta)^2+\varepsilon^2}{3A\varepsilon}.
\]

\textbf{Step 4.}
We consider the case $m=1$, i.e., $a=b=1$ and $A=2-\delta$. Since $Y$ is the 
del Pezzo surface of degree $7$, we can easily show that 
\[
\vol(L-xF)=\begin{cases}
2(2-\delta)^2-x^2 & \text{if }x\in[0, 2-\delta], \\
\left(2(2-\delta)-x\right)^2 & \text{if } x\in[2-\delta, 2(2-\delta)]. 
\end{cases}\]
Thus we get the equality $\hat{\beta}_{(X, \Delta)}(F)=0$. In fact, by Proposition 
\ref{PP_prod_prop}, the divisor $F$ is a product-type prime divisor over $(X, \Delta)$. 
In particular, we have proved Theorem \ref{FF_thm} \eqref{FF_thm2}. 

\textbf{Step 5.}
We consider the case $m\geq 2$. Assume that $p_2\not\in C^{X_1}$. 
Then we can inductively show that $A_i=a_i+b_i-b_i\delta$. In particular, we have 
$A=a+b-b\delta$. By Step 2, we have 
\[
\hat{\beta}_{(X, \Delta)}(F)\geq 1-\frac{2-\delta}{2}\cdot\frac{a+b}{a+b-b\delta}
=\frac{(a-b)\delta}{2(a+b-b\delta)}>0. 
\]
Thus we may assume that $m\geq 2$ and $p_2\in C^{X_1}$. 

\textbf{Step 6.}
Let us set 
\[
j_C:=\max\{2\leq i\leq k\,\,|\,\, E_i\cap C^{X_i}\neq \emptyset\}. 
\]
Then we can inductively show that 
\[
\pi^*C=\begin{cases}
C^{\tilde{X}}+\sum_{i=1}^{j_C}i\tilde{E}_i+\sum_{i=j_C+1}^m j_C b_i\tilde{E}_i & 
\text{if }j_C<k, \\
C^{\tilde{X}}+\sum_{i=1}^ma_i\tilde{E}_i & 
\text{if }j_C=k.
\end{cases}\]
As in the argument in Step 7 for the proof of Theorem \ref{F_thm}, we have 
$A=a+b-\min\{j_Cb$, $a\}\delta$. If $2j_C\leq k$, then, from Step 2, we have 
\begin{eqnarray*}
\hat{\beta}_{(X, \Delta)}(F)&\geq& 1-\frac{2-\delta}{2}\cdot\frac{a+b}{a+b-j_Cb\delta}\\
&=&\frac{\delta(a+b-2j_Cb)}{2(a+b-j_Cb\delta)}
\geq\frac{\delta(a-(k-1)b)}{2(a+b-j_Cb\delta)}>0.
\end{eqnarray*}
Thus we may further assume that $2j_C>k$. 

\textbf{Step 7.}
Assume that $k=2$. Then we have $2b\geq a$ and $A=a+b-a\delta$. By 
Lemma \ref{plt_lem} \eqref{plt_lem4}, we have 
\[
\left(\left(\pi^*L-\frac{x}{f}E_m^*\right)\cdot l_1^{\tilde{X}}\right)=2-\delta-\frac{x}{a}.
\]
Thus we have $\varepsilon\leq a(2-\delta)$. Since $C^{\tilde{X}}$ is nef and 
$\pi^*C=C^{\tilde{X}}+\sum_{i=1}^ma_i\tilde{E}_i$, 
we have $\varepsilon=a(2-\delta)$ by Lemma \ref{star_lem} \eqref{star_lem5}. 
By Step 3, we get 
\[
\hat{\beta}_{(X, \Delta)}(F)\geq 1-\frac{2ab(2-\delta)^2
+a^2(2-\delta)^2}{3(a+b-a\delta)a(2-\delta)}=\frac{(1-2\delta)(a-b)}{3(a+b-a\delta)}\geq 0. 
\]
When $\delta<1/2$, then we get $\hat{\beta}_{(X, \Delta)}(F)>0$. 

\textbf{Step 8.}
Thus we may further assume that $k\geq 3$. By Lemma \ref{plt_lem} \eqref{plt_lem4}, 
we have 
\[
\left(\left(\pi^*L-\frac{x}{f}E_m^*\right)\cdot C^{\tilde{X}}\right)=2(2-\delta)
-\frac{x}{ab}\min\{j_Cb,\,\, a\}.
\]
Thus we get 
\[
\varepsilon\leq \frac{2ab(2-\delta)}{\min\{j_Cb,\,\, a\}}.
\]
Assume that $j_C^2b\geq 2a$. Then, by the assumption $k\geq 3$, 
Lemma \ref{star_lem} \eqref{star_lem5} and Step 6, we have 
\[
\varepsilon=\frac{2ab(2-\delta)}{\min\{j_Cb,\,\, a\}}.
\]
Therefore, by Step 3, we have 
\[
\hat{\beta}_{(X, \Delta)}(F)\geq 
1-\frac{2ab(2-\delta)^2+\left(\frac{2ab(2-\delta)}{\min\{j_Cb,\,\,a\}}
\right)^2}{3(a+b-\min\{j_Cb,\,\,a\}\delta)\frac{2ab(2-\delta)}{\min\{j_Cb,\,\,a\}}}. 
\]
If $j_C=k$, then we have 
\[
\hat{\beta}_{(X, \Delta)}(F)\geq\frac{(1-2\delta)(a-b)}{3(a+b-a\delta)}\geq 0.
\]
When $j_C=k$ and $\delta<1/2$, we have $\hat{\beta}_{(X, \Delta)}(F)>0$. 
If $j_C<k$, then we have 
\begin{eqnarray*}
\hat{\beta}_{(X, \Delta)}(F)&\geq&\frac{3j_C(a+b)-2(2a+j_C^2b)
+2(a-j_C^2b)\delta}{3j_C(a+b-j_Cb\delta)}\\
&\geq&\frac{3j_C(a+b)-2(2a+j_C^2b)+a-j_C^2b}{3j_C(a+b-j_Cb\delta)}\\
&=&\frac{(j_C-1)(a-j_Cb)}{j_C(a+b-j_Cb\delta)}>0.
\end{eqnarray*}

\textbf{Step 9.}
Thus we can further assume that $j_C^2b<2a$. Since we have already assumed that 
$2j_C>k$, we get 
\[
\frac{j_C^2}{2}<\frac{a}{b}<2j_C. 
\]
This implies that $(j_C$, $k)=(2$, $3)$ or $(3$, $5)$. Moreover, if $(j_C$, $k)=(3$, $5)$, 
then we may assume that $a/b>9/2$. Let $\lambda\colon X_1\to\pr^2$ be the 
birational morphism contracting $l_1^{X_1}$ and $l_2^{X_1}$. 

\textbf{Step 10.}
Assume that $(j_C$, $k)=(2$, $3)$. Then we can uniquely find the line $l$ on $\pr^2$ 
with $l^{X_3}\cap E_3\neq \emptyset$. Since $a\leq 3b$, we have 
\[
\left(\left(\pi^*L-\frac{x}{f}E_m^*\right)\cdot l^{\tilde{X}}\right)=2(2-\delta)-\frac{x}{b}
\]
by Lemma \ref{plt_lem} \eqref{plt_lem4}. Thus we have $\varepsilon\leq 2b(2-\delta)$. 
Moreover, we can inductively show that 
\[
\pi^*\left(l^{X_0}\right)=l^{\tilde{X}}+\sum_{i=1}^ma_i\tilde{E}_i. 
\]
By Lemma \ref{star_lem} \eqref{star_lem5}, we get $\varepsilon=2b(2-\delta)$. 
Thus, from Step 3, we have 
\begin{eqnarray*}
\hat{\beta}_{(X, \Delta)}(F)&\geq&
1-\frac{2ab(2-\delta)^2+\left(2b(2-\delta)\right)^2}{3(a+b-2b\delta)\cdot 2b(2-\delta)}\\
&=&\frac{a-b+\delta(a-4b)}{3(a+b-2b\delta)}
\geq\frac{a-2b}{2(a+b-2b\delta)}>0. 
\end{eqnarray*}

\textbf{Step 11.}
Assume that $(j_C$, $k)=(3$, $5)$ and $a/b>9/2$. It is well-known that there exists a 
unique smooth conic $D$ on $\pr^2$ such that $D^{X_5}\cap l_1^{X_5}\neq \emptyset$ 
and $D^{X_5}\cap E_5\neq\emptyset$. By Lemma \ref{plt_lem} \eqref{plt_lem4}, we have 
\[
\left(\left(\pi^*L-\frac{x}{f}E_m^*\right)\cdot D^{\tilde{X}}\right)=3(2-\delta)-\frac{x}{b}.
\]
Thus we get $\varepsilon\leq 3b(2-\delta)$. On the other hand, we know that 
\[
\pi^*C=C^{\tilde{X}}+\tilde{E}_1+2\tilde{E}_2+\sum_{i=3}^m3b_i\tilde{E}_i
\]
and 
\[
\left(\left(\pi^*C-\frac{3}{a}E_m^*\right)\cdot C^{\tilde{X}}\right)=2-\frac{9b}{a}>0.
\]
Thus we get $\varepsilon=3b(2-\delta)$ by Lemma \ref{star_lem} \eqref{star_lem5}. 
Hence we have 
\begin{eqnarray*}
\hat{\beta}_{(X, \Delta)}(F)&\geq&
1-\frac{2ab(2-\delta)^2+\left(3b(2-\delta)\right)^2}{3(a+b-3b\delta)\cdot 3b(2-\delta)}\\
&=&\frac{5a-9b+2\delta(a-9b)}{9(a+b-3b\delta)}\geq \frac{2(a-3b)}{3(a+b-3b\delta)}>0.
\end{eqnarray*}

As a consequence, we have completed the proof of Theorem \ref{FF_thm}. 
\end{proof}

\section{Proof of Corollary \ref{maincor}}\label{maeda_section}

In this section, we prove Corollary \ref{maincor}. 
We recall the result of Maeda. 
We set $\F_m:=\pr_{\pr^1}\left(\sO\oplus\sO(m)\right)$ ($m\geq 0$) and let 
$e\subset\F_m$ be a section of $\F_m\to\pr^1$ with the self intersection number $-m$, 
let $l\subset\F_m$ be a fiber of $\F_m\to\pr^1$, and let $e_\infty\subset\F_m$ 
is a section of $\F_m\to\pr^1$ with the self intersection number $m$. 

\begin{thm}[{\cite{maeda}}]\label{maeda_thm}
Let $X$ be a smooth projective surface and let $D$ be a nonzero effective reduced 
simple normal crossing divisor on $X$ with $-(K_X+D)$ ample. Then $(X, D)$ is 
isomorphic to one of 
$(\pr^2$, line$)$, 
$(\pr^2$, the union of two distinct lines$)$,
$(\pr^2$, smooth conic$)$, 
$(\pr^1\times\pr^1$, diagonal$)$, 
$(\F_1$, $e_\infty)$, 
$(\F_m$, $e)$, or 
$(\F_m$, $e+l)$.
\end{thm}

Corollary \ref{maincor} is an immediate consequence of Theorems \ref{F_thm}, 
\ref{FF_thm}, \ref{maeda_thm} and \cite[Theorem 1.2]{BB} for example. 
We give an elemental proof of Corollary \ref{maincor} for the readers' convenience. 

\begin{proof}[Proof of Corollary \ref{maincor}]
Assume that $X=\F_m$ and $\Delta=\delta_1 e+\delta_2 l$ with 
$\delta_1\in(0$, $1)\cap\Q$ and $\delta_2\in[0$, $1)\cap\Q$. Then the pair 
$(X, \Delta)$ is a log del Pezzo pair if and only if $m+2-\delta_2>m(2-\delta_1)$. 
The $\R$-divisor $L-xe\sim_\R(2-\delta_1-x)e+(m+2-\delta_2)l$ for $x\in\R_{\geq 0}$ is 
nef if and only if $x\leq 2-\delta_1$. Thus we have 
\begin{eqnarray*}
\hat{\beta}_{(X, \Delta)}(e)&=&1-\frac{\int_0^{2-\delta_1}
((L-xe)^{\cdot 2})dx}{A_{(X, \Delta)}(F)
(L^{\cdot 2})}\\
&=&\frac{2m\delta_1-2m\delta_1^2-6\delta_1+3\delta_1\delta_2
-2m}{3(1-\delta_1)(m\delta_1+4-2\delta_2)}.
\end{eqnarray*}
If $m=0$, then we can immediately show that $\hat{\beta}_{(X, \Delta)}(e)<0$; 
if $m\geq 1$, then 
$\hat{\beta}_{(X, \Delta)}(e)<0$ since 
\begin{eqnarray*}
&&2m\delta_1-2m\delta_1^2-6\delta_1+3\delta_1\delta_2-2m\\
&=&-2m\left(\delta_1-\frac{2m-6+3\delta_2}{4m}\right)^2\\
&&-\frac{3\left(2m+3(2-\delta_2)\right)\left(2m-(2-\delta_2)\right)}{8m}<0.
\end{eqnarray*}

Assume that $X=\F_1$ and $\Delta=\delta e_\infty$ with $\delta\in[0$, $1)\cap\Q$. 
The $\R$-divisor $L-xe\sim_\R(2-\delta-x)e+(3-\delta)l$ for $x\in\R_{\geq 0}$ is 
nef if and only if $x\leq 2-\delta$. Thus we have 
\[
\hat{\beta}_{(X, \Delta)}(e)=\frac{-2(1-4\delta+\delta^2)}{3(4-\delta)}.
\]
If $\delta<2-\sqrt{3}$, then $\hat{\beta}_{(X, \Delta)}(e)<0$. Similarly, 
The $\R$-divisor $L-xe_\infty\sim_\R(2-\delta-x)e+(3-\delta-x)l$ for $x\in\R_{\geq 0}$ 
is nef if and only if $x\leq 2-\delta$. Thus we have 
\[
\hat{\beta}_{(X, \Delta)}(e_\infty)=\frac{2(1-4\delta+\delta^2)}{3(4-\delta)(1-\delta)}.
\]
If $\delta>2-\sqrt{3}$, then we have $\hat{\beta}_{(X, \Delta)}(e_\infty)<0$. 
Since $\delta\in\Q$, 
the pair $(X, \Delta)$ is not K-semistable for any $\delta\in[0$, $1)\cap\Q$. 

Assume that $X=\pr^2$ and $\Delta=\delta_1 l_1+\delta_2 l_2$ with 
$l_1$, $l_2$ distinct lines, $\delta_1$, $\delta_2\in[0$, $1)\cap\Q$, $\delta_1\leq \delta_2$  
and $(\delta_1$, $\delta_2)\neq(0$, $0)$. Then we can immediately get the 
inequality
\[
\hat{\beta}_{(X, \Delta)}(l_2)=\frac{-\delta_2-\left(\delta_2-\delta_1\right)}{3(1-\delta_2)}<0. 
\]

Together with Theorems \ref{F_thm} and \ref{FF_thm}, we get the assertion. 
\end{proof}

\end{document}